\documentclass[reqno]{amsart}
\usepackage{graphicx,graphics} 
\usepackage{epstopdf,epsfig}
\DeclareGraphicsExtensions{.eps}
\usepackage[all]{xy}
\usepackage{textcomp,gensymb}
\usepackage{eurosym}
\usepackage{enumerate}
\usepackage[plainpages=false,pdfpagelabels,colorlinks, 
hypertexnames=false]{hyperref}
\hypersetup{citecolor=blue}
\usepackage{caption}
\usepackage{subcaption}
\usepackage{amsmath,amsthm, amscd, amssymb, amsfonts, mathrsfs,graphics}
\usepackage[mathscr]{euscript} 
\usepackage[normalem]{ulem}
\newtheorem{theorem}{Theorem}[section]

\newtheorem{lemma}[theorem]{Lemma}
\newtheorem{proposition}[theorem]{Proposition} 
 
\theoremstyle{definition}

\theoremstyle{remark}

\newtheoremstyle{mystyle}{2mm}{0mm}{}{}{\bfseries}{}{1ex}{\thmnumber{#2}.\hspace*{1ex}\thmnote{#3}}
\theoremstyle{mystyle}
\newtheorem{fact}[theorem]{}
\newtheoremstyle{myremark}{2mm}{0mm}{}{}{\bfseries}{}{ }{\thmname{#1}
\thmnumber{#2}. \thmnote{ #3}}
\theoremstyle{myremark}
\newtheorem{remark}[theorem]{Remark}

\newcommand{\bs}[1]{\boldsymbol{#1}}

\def\Id{\mathrm{I}} 
\DeclareMathOperator{\Ker}{Ker} 

\DeclareMathOperator{\Ho}{H}

\DeclareMathOperator{\K}{K}

\DeclareMathOperator{\Tor}{Tor}

\newcommand\ot{\otimes}

\newcommand\op{\oplus}
\newcommand\kal{\mathscr}

\newcommand\dr{\mathrm}

\newcommand{\Ap}[1]{A_+^{(#1)}}

\newcommand{\br}[1]{\{ #1 \}}

\def\T3{\mathbb{T}_3}

\def\m{\mu}

\newcommand{\Od}[2]{\Omega_{#1}(#2)}

\newcommand{\N}{{\mathbb N}}

\newcommand{\pa}{\partial}

\DeclareMathAlphabet{\mathpzc}{OT1}{pzc}{m}{it}

\begin{document}
\title{On Koszulity of Finite Graded Posets}
\author{Adrian Manea}
\author{Drago\c{s} \c{S}tefan}
\address{University of Bucharest, Department of Mathematics, 14 Academiei Street, Bucharest Ro-010014, Romania}
\subjclass[2010]{Primary 16E40; Secondary 16T10 and 16T15}
\keywords{Koszul rings, Koszul corings, incidence (co)ring of a poset}
\date{}

\begin{abstract}
In this paper we continue our research on Koszul rings, started in \cite{jps}. In Theorem \ref{thm:kring} we prove in a unifying way several equivalent descriptions of Koszul rings, some of which being well known in the literature. Most of them are stated in terms of coring theoretical properties  of $\Tor_n^A(R,R)$. As an application of these characterizations we investigate the Koszulity of the incidence rings for finite graded posets, see Theorem \ref{te:aplicatie} and Theorem \ref{te:aplicatie_dual}. Based on these  results, we describe an algorithm to produce new classes of Koszul posets (i.e. graded posets whose incidence rings  are Koszul). Specific examples of Koszul posets are included.
\end{abstract}

\keywords{Koszul ring, incidence ring of a poset}
\subjclass[2010]{Primary 16E40; Secondary 16T15}

\maketitle
  
\section*{Introduction}
By the definition in \cite{bgs}, an $\mathbb{N}$-graded
ring $A:=\oplus_{n\in\mathbb{N}}A_n$ is said to be Koszul if and only if $A_{0}$ is a semisimple ring and there is a resolution $P_{\ast}$ of $A_0$ by projective
graded left $A$-modules such that each $P_{n}$ is
generated by homogeneous elements of degree $n$. 
Koszul rings are natural generalizations of Koszul algebras which, in turn, were discovered by Priddy \cite{Pr}.  Koszul algebras and rings have proved to be very useful tools in various fields of Mathematics, such as: Representation Theory, Algebraic Geometry, Algebraic Topology, Quantum Groups and Combinatorics; a comprehensive read is, for example, \cite{PP} and the references therein.

In this paper we continue our previous work on Koszul rings and their applications, our purpose being two-fold. First, we show that the Koszulity of a connected graded $R$-ring $A$ can be stated  in terms of coring theoretical properties of its connected graded coring $T(A):=\Tor_\ast^A(R,R)$.  This is mainly done  in Theorem \ref{thm:kring}, where we also recover in a unifying way other well-known equivalent characterizations of these rings. For instance, we show that $A$ is Koszul if and only if the coring $T(A)$ is strongly graded, if and only if the primitive part of $T(A)$ coincides with its homogeneous component of degree one, if and only if $T(A)$, regarded in a canonical way as a bigraded coring, is diagonal (for the last characterization see also \cite{wood}). 

To check Koszulity of a ring is a difficult task. The case of incidence rings of finite graded posets was approached in \cite[Theorem 3.7]{wood}, where it is proved that such a ring is Koszul if and only if all open intervals of the poset are Cohen-Macaulay. The Koszulity of incidence algebras is  investigated in \cite{polo, RS} as well. Nevertheless, the examples of Koszul posets are sporadic.   

Our second goal is to indicate an algorithmic method for producing new examples of Koszul posets.   This is done in the second section of our paper. To explain the main result of this part, let us introduce some notation. We fix a finite graded poset $\kal{P}$ and a maximal element $t\in \kal{P}$. We denote by $\kal{Q}$ the subposet  $\kal{Q}=\kal{P}\setminus\{t\}$. For the set of predecessors of $t$ in $\kal{P}$ we will use the notation $\kal{F}$. Let $B$ be the incidence ring of $\kal{Q}$.  In  Theorem \ref{te:aplicatie}, under mild combinatorial conditions on the set $\kal{F}$, we prove that $A$ is  Koszul if and only if $B$ is Koszul. A similar result is obtained in Theorem \ref{te:aplicatie_dual}, working with dual poset of $\kal{P}$.
Based on these two theorems we provide an algorithm to construct new examples of Koszul graded posets from given ones. 
We use this algorithmic method to distinguish several concrete types of Koszul posets, such as planar tilings and nested diamonds.

\section{Koszul rings}
In this section we first recall some basic concepts and notations and then we will prove some new preliminary results, which are needed later on.  The main result of this section is Theorem \ref{thm:kring}, where we characterize  Koszul rings.

\begin{fact}[Connected (co)rings.]                                                                                                                                                                                                                                                                                                                                                                                                                                                                                                                                                                                                                                                                                                                                                                                                                                                                                                                                                               Throughout, $R$ will denote a semisimple ring. Since we always work with algebras and coalgebras in the tensor category of $R$-bimodules, to ease the notation, the tensor product $\ot_R$ of two bimodules will be denoted by $\ot$. If $V$ is an $R$-bimodule, then we will use the notation $V^{(n)}$ for the $n$th tensor power $V\ot\cdots\ot V$. By convention we take $V^{(0)}=R$. In a similar way, for any bimodule morphism $f:M\to N$, the tensor product $f\ot\cdots\ot f$ with $n$ factors will be denoted by $f^{(n)}$. The identity morphism of $X$ will be denoted either by $\Id_X$, or simply $X$, if there is no danger of confusion.
	
By definition $A=\oplus_{n \in \mathbb{N}} \ A_n$ is a connected $R$-ring if and and only if it is an $\mathbb{N}$-graded algebra in the tensor category of $R$-bimodules  and  $A_0=R$. 
Since $A$ is graded, the multiplication $\mu:A\ot A\to A$ is defined by some maps $\mu_{p,q} : A_p \otimes A_q \rightarrow A_{p+q}$. We say that $A$ is \emph{strongly graded} if these maps are surjective, for all $p,\ q \geq 0$. Equivalently, $A$ is strongly graded if and only if the iterated multiplication $\m(n):A_1^{(n)}\to A_n$ is surjective for all $n\geq 2$.

Dually, a connected $R$-coring is a graded coalgebra  $C=\oplus_{n \in \mathbb{N}} C_n$, in the category of $R$ bimodules such that $C_0=R$. 
The comultiplication of a connected $R$-coring $C$ is defined by some $R$-bimodule maps $\Delta_{p,q} : C_{p+q} \rightarrow C_p \otimes C_q$ and the coring $C$ is called \emph{strongly graded} if all of them are injective.
	
Let $\Delta(1):=\Id_{C_1}$ and for $n\geq 2$ we define $\Delta(n) : C_n \rightarrow C_1^{(n)}$ by the recurrence relation:
\[
\Delta(n)=\big({C_1} \otimes \Delta(n-1)\big) \circ \Delta_{1,n-1}.
\]
From the very definition of the maps $\Delta(n)$ and coassociativity, we get the equality:
\begin{equation}\label{eq:Delta(n)}
 \Delta(p+q)=\big(\Delta(p) \otimes \Delta(q)\big) \circ \Delta_{p,q}.
\end{equation}
Note that $C$ being strongly graded can be put in terms of $\Delta(n)$ being injective for all $n$. Equivalently, $C$ is strongly graded if and only if $\Delta_{1,n}$ is injective for all $n$, if and only if $\Delta_{n,1}$ in injective for all $n$.
	
Let $C$ be a connected $R$-coring, so $C_0=R$. By definition of connected corings the unit of $R$ is a group-like element, that is $\Delta(1)=1\ot 1$, see \cite{jps}.  Hence it makes sense to speak about the  primitive elements of $C$, i.e. about those $c \in C$ such that $\Delta(c)=c\otimes 1 + 1 \otimes c$.  The set of primitive elements in $C$ includes $C_1$ and it will be denoted by $P(C)$. In general, the inclusion $C_1\subseteq P(C)$ is strict. 
\end{fact} 

\begin{fact}[The $R$-coring $A^!$.]\label{fact:shrieck}
For an $R$-bimodule $V$, the tensor $R$-coring is the graded $R$-bimodule whose component of degree $n$ is $T_R^c(V)_{n}=V^{\ot n}$.  The comultiplication of   $T_R^c(V)$ is defined by the family $\{\Delta_{p,q}\}_{p,q\in\N}$, where:  
\[                                                                                                                                                                                                          
\Delta_{p,q}:T_R^c(V)_{p+q}\to T_R^c(V)_p\ot T_R^c(V)_q, \quad \Delta_{p,q}(v_1\ot\cdots \ot v_{p+q})=(v_1\ot\cdots\ot v_p)\ot (v_{p+1}\ot\cdots\ot  v_{p+q}).                                                                                                                                                                                                                                           \]
For a sub-bimodule $W\subseteq V\ot V$, we define  a graded $R$-module $\br{V,W}=\oplus_{n\in\N}\br{V,W}_n$ by taking $\br{V,W}_0=R$ and  $\br{V,W}_1=V$. For all $n \geq 2$ one defines:
\[ 
\br{V,W}_n=\displaystyle\bigcap_{p=1}^{n-1} V^{(p-1)} \otimes W \otimes V^{ (n-p-1)}.
\]
As it is shown in  \cite{jps}, $\br{V,W}$ is a graded subcoring of $T^c_R(V)$.  

If $A$ is a connected graded $R$-ring then $\{A_1,\Ker\mu_{1,1}\}$  will be denoted by $A^!$. The homogeneous component of degree $n$ of $A^!$ will be denoted by $A^!_n$. 

Since $A^!$ is an $R$-subcoring of $T_R^c(A_1)$, the map $\Delta_{1,n-1}$ induces a morphism $\eta_n:A^!_n\to A_1\ot A^!_{n-1}$, for every $n>0$.
\end{fact}

\begin{fact}[Bigraded corings.]
We collect here some basic facts regarding bigraded corings, which hold in general and which we will be using throughout this article when discussing quadratic and Koszul $R$-corings.

Let $C$ be an $R$-coring. We say that $C$ is \emph{bigraded} if $C$ has a decomposition as a direct sum of $R$-bimodules $C=\bigoplus_{n,m \in \mathbb{N}} \ C_{n,m}$ such that its comultiplication induces a collection of maps
\[ C_{n+n',m+m'} \xrightarrow{\Delta_{n,m}^{n',m'}} C_{n,m} \otimes C_{n',m'}.\]
In this context, coassociativity translates as the commutativity of the diagram below, for all positive integers $m$, $n$, $p$, $m'$, $n'$ and $p'$.
\[
 \xymatrixcolsep{3pc}\xymatrix{ C_{n+m+p,n'+m'+p'} \ar[d]_-{\Delta_{n,n'}^{m+p,m'+p'}} \ar[r]^-{\Delta_{n+m,n'+m'}^{p,p'}} & C_{n+m,n'+m'} \otimes C_{p,p'} \ar[d]^-{\Delta_{n,n'}^{m,m'} \otimes\, \Id_{C_{p,p'}}} \\
C_{n,n'} \otimes C_{m+p,m'+p'} \ar[r]_-{\Id_{C_{n,n'}} \otimes \Delta_{m,m'}^{p,p'}} & C_{n,n'} \otimes C_{m,m'} \otimes C_{p,p'}}
\]
By definition, the counit must vanish on $C_{n,m}$, provided that either $n>0$ or $m>0$.

To any bigraded coring, one can associate a graded coring $\mathrm{gr}(C)$, whose homogeneous component of degree $n$ is $\mathrm{gr}_n(C)=\oplus_{m\in\N} \ C_{n,m}$. Thus, we define $\mathrm{gr}(C):=\oplus_{n\in\N} \ \mathrm{gr}_n(C)$. In this paper we are interested only in \textit{connected bigraded corings}, i.e. in those bigraded corings for which  $C_{0,0}=R$ and $C_{0,m}=0$, for $m>0$. Note that, if $C$ is connected, then $\mathrm{gr}(C)$ is connected as well.

If $C$ is a (connected) bigraded coring and we put:
\[ C'_{n,m}=\begin{cases} C_{n,m}, & n=m; \\ 0, & n\neq m,\end{cases}\]
then ${C'}:=\oplus_{n,m\in\N} {C}_{n,m}'$ is a (connected)  bigraded coring. We denote the graded coring  $\mathrm{gr}({C'})$ by $\mathrm{Diag}(C)$. 

For $C$ and $C'$ as above, there are canonical $R$-bimodule morphisms $\pi_{n,m}:C_{n,m} \rightarrow {C}_{n,m}'$ which act as identity on $C_{n,n}$ and zero maps in rest. The collection $\pi=\{ \pi_{n,m} \}_{n,m\in\N}$ defines a morphism of bigraded corings. As $\mathrm{Diag}(C)=\mathrm{gr}({C'})$, the map $\pi$ induces a morphism $\mathrm{gr}(\pi): \mathrm{gr}(C) \xrightarrow{} \mathrm{Diag}(C)$ of graded corings. Clearly, the $n$-degree component of the kernel of $\pi$ coincides with $\oplus_{n\neq m} C_{n,m}$. 
\end{fact}

\begin{lemma} \label{lemain}
Let $C$ be a connected coring.
\begin{enumerate}

\item The coring $C$ is strongly graded if and only if $P(C)=C_1$.
 
\item   If $C$ is strongly graded and $f : C \rightarrow D$ is a  morphism of graded corings such that the components $f_0$ and $f_1$ are injective, then $f$ is injective. 
	 	 
\item Let $C=\oplus_{n,m\in\N}C_{n,m}$ be a bigraded $R$-coring. If $\mathrm{gr}(C)$ is strongly graded and $C_{n,m}=0$ for $n=0,1$ and all $m\neq n$, then the component $C_{n,m}$ is also trivial, for all $n\geq 2$ and $m\neq n$. 
\end{enumerate}
\end{lemma}
	
\begin{proof}
Let us prove the first part of the lemma. We first assume that  $C$ is strongly graded. Let $x=\sum_{k=0}^n x_k$ be a primitive element of $C$, where each $x_k$ is homogeneous of degree $k$. Since $C$ is graded, $\Delta(x_k)= \sum_{p=0}^{k}  \Delta_{p,k-p}(x_k)$.  Then
\[
\sum_{k=0}^n x_k \ot 1 + \sum_{k=0}^n 1 \ot x_k =x \ot 1 + 1 \ot x = \Delta(x)=  \sum_{k=0}^n \sum_{p=0}^{k}  \Delta_{p,k-p}(x_k).
\]
We claim that $x_k=0$ for any $k\neq 1$. Because $x_0\in C_0$ we have $\Delta(x_0)=\Delta_{0,0}(x_0)=x_0 \ot 1$, so by equating the terms in $C_0\ot C_0$, it follows that $x_0\ot 1=0$. Thus  $x_0=0$.
	
Let us fix $k \geq 2$. In the rightmost term of the above equation, $\Delta_{p,k-p}(x_k) \in C_p \ot C_{k-p}$. On the other hand, $x_k \ot 1 + 1 \ot x_k \in C_k\ot C_0 + C_0 \ot C_k$, for any $k$. Hence, $\Delta_{p,k-p}(x_k)=0$. Since $C$ is strongly graded, all $\Delta_{p,k-p}$ are injective. It follows that $x_k=0$, so our claim was proved. In conclusion, $x=x_1 \in C_1$, that is we have $P(C)\subseteq C_1$. The other inclusion always holds,  completing the proof of the direct implication.

Note that for $k>1$, by the above proof, we get the relation $P(C)\bigcap C_k=\bigcap_{p=1}^{k-1}\Ker\Delta_{p,k-p}$.

Conversely, let us assume that $P(C)= C_1$ and prove that all maps $\Delta{(n)}$ are injective. Since  $\Ker \Delta(2)=\Ker \Delta_{1,1}$  coincides with $P(C)\bigcap C_2=0$, it follows that this map is injective. We assume that $\Delta(k)$ is injective for all $k\leq n$. Hence, the map $\Delta(p)\otimes\Delta(n+1-p)$ is injective, for all $0<p<n+1$. Using the relation \eqref{eq:Delta(n)} we deduce that $\Delta_{p,n+1-p}(x)=0$, for any $x\in \Ker\Delta(n+1)$ and all $p$ as above. Thus $x\in P(C)\bigcap C_{n+1}=0$, that is $\Delta(n+1)$ is injective.

To prove (2) we assume, for an inductive argument, that the component  $f_k$ is injective for  $k=0,\dots, n$. As $f$ is a morphism of graded corings, the diagram: 
\[
 \xymatrixcolsep{2pc}\xymatrix{ C_{n+1} \ar[d]_-{f_{n+1}} \ar[r]^-{\Delta_{n,1}^{C}} & C_{n} \otimes C_{1} \ar[d]^-{f_n \otimes\, f_1} \\
D_{n+1} \ar[r]_-{\Delta_{n,1}^{D}} &  D_{n} \otimes D_{1}}
\]
is commutative. By the standing assumptions, $\Delta_{n,1}^C$ and $f_n\ot f_1$ are injective.  It follows that $f_{n+1}$ is injective. Thus, by induction, all components of $f$ are injective, so $f$ itself is injective.

It remains to prove the last part of the lemma. Let $\mathrm{Diag}\,C=\oplus_{n\in\N}C_{n,n}$. We denote the canonical morphism of graded $R$-corings by $\mathrm{gr}{(\pi)}: \mathrm{gr}(C) \rightarrow \mathrm{Diag}\,C$. Since we have $\mathrm{gr}_1(C)=\oplus_{m\in\N} C_{1,m}=C_{1,1}$, the component $\mathrm{gr}_1(\pi)$ coincides with the identity map of $ C_{1,1}=P(\mathrm{gr}(C))=\mathrm{Diag}_1(C)$. On the other hand, as $C_{0,0}=\mathrm{gr}_0(C)=\mathrm{Diag}_0(C)$, the component $\mathrm{gr}_0(\pi)$  coincides with the identity map of $C_{0,0}$. Using the first part of the lemma, it follows that $\mathrm{gr}(\pi)$ is injective. As $\mathrm{Ker}\, \mathrm{gr}_n(\pi) =\oplus_{m \neq n} C_{n,m}$, we get $C_{n,m}=0$, for all  $n$ and $m$, with $m\neq n$.
\end{proof}
	
\begin{remark}
In the particular case of ordinary coalgebras, the first part of the lemma is also proved in \cite[Lemma 2.3]{MS}.
\end{remark}

\begin{fact}[The normalized bar resolution.] \label{sec-t(a)}
The groups $\Tor_{\ast }^{A}(R,R)$ may be computed using the \emph{normalized bar resolution} $\beta_{\ast}(A)$, that is the complex given by $\beta_n(A)=A\otimes A_+^{( n)}$, for all $n\geq 0$. The differential map  $\delta_n$ is defined, for every $n>0$, by  the relation:
\begin{equation*}
\delta_{n}(a_{0}\otimes \cdots \otimes a_{n-1}\otimes a_{n}):=\sum\limits\limits_{i=0}^{n-1}(-1)^{i}a_{0}\otimes \cdots \otimes a_{i}a_{i+1}\otimes \cdots \otimes a_{n}.
\end{equation*}%
By  \cite[p. 283]{wei}, the $n$th homology group of $\beta_\ast(A)$ is trivial, if $n>0$, but $\Ho_0(\beta_\ast(A))=R$. 
Since $R$ is semisimple, it follows that $\beta_\ast(A)$ is a resolution of $R$ by projective left $A$-modules. 

Hence the homology groups of the {\it normalized bar complex} $\Omega_\ast(A)=R\ot_A\beta_\ast(A)$ coincide with $\Tor_\ast^A(R,R)$. Obviously, $\Omega_n(A)=\Ap{n}$ and $d_1=0$. For $n>1$, the differential morphism $d_n:\Omega_n(A)\to\Omega_{n-1}(A)$ satisfies the relation:
\[
d_n(a_1\ot\cdots \ot a_n)=\sum_{i=1}^{n-1}(-1)^{i-1}a_1\ot\cdots \ot a_ia_{i+1}\ot\cdots\ot a_n.
\]
Since the normalized bar complex has a canonical structure of DG-coalgebra in the category of $R$-bimodules with respect to the comultiplication $\Delta_{p,q}(a_1\ot\cdots \ot a_{p+q})=(a_1\ot\cdots \ot a_{p})\ot(a_{p+1}\ot\cdots \ot a_{p+q})$ it follows that $T(A)=\oplus_{n \in \mathbb{N}} \ T_n(A)$ has a canonical structure of connected $R$-coring. Note that $\Od{\ast}{A}$ and $T_R^c(A_+)$ are isomorphic as connected $R$-corings.

The complex $\Omega_\ast(A)$ decomposes as  a direct sum  $\Omega_\ast(A)=\oplus_{m \geq 0}\,\Omega_\ast(A,m)$ of subcomplexes. In order to define  $\Omega_\ast(A,m)$ we introduce the  following terminology and notation. An $n$-tuple $\boldsymbol{m}=(m_1,\dots,m_n)$ is called \textit{a positive $n$-partition of $m$} if and only if $\sum_{i=1}^n m_i=m$ and all $m_i$ are positive. The set of all positive $n$-partitions of $m$ will be denoted by $\mathscr{P}_n(m)$. Furthermore, if $\boldsymbol{m}=(m_1,\dots,m_n)$ is a positive $n$-partition and $A$ is a connected $R$-ring then for the tensor product $A_{m_1}\ot\cdots\ot A_{m_n}$ we will use the notation $A_{\boldsymbol{m}}$. 
For a positive $n$-partition $\bs{m}=(m_1,\dots,m_n)$ of $m$, the multiplication $\m$ of $A$ induces bimodule maps $\mu_{\bs{m}}:A_{\bs{m}}\to A_m$ and $\mu(\bs{m}):A_1^{(m_1)}\ot\cdots \ot A_1^{(m_n)}\to A_{\bs{m}}$. Note that, by definition,  $\mu(\bs{m})=\m(m_1)\ot\cdots\ot\m(m_n)$.

Hence, with this notation, $\Omega_\ast(A,m)$ is the following subcomplex of $\Omega_\ast(A)$:
\[
\ 0 \xleftarrow{d_0^m} 0 \xleftarrow{d_1^m}\textstyle \bigoplus \limits_{\boldsymbol{m_1}\in\mathscr{P}_1(m)}  A_{\boldsymbol{m_1}} \xleftarrow{d_2^m}\textstyle\textstyle \bigoplus \limits_{\boldsymbol{m_2}\in\mathscr{P}_2(m)} A_{\boldsymbol{m_2}}\xleftarrow{d_3^m} \cdots \xleftarrow{d^m_n} \textstyle \bigoplus \limits_{\boldsymbol{m_n}\in\mathscr{P}_{n}(m)} A_{\boldsymbol{m_n}}\xleftarrow{\ \ \ }\cdots .
\]
Of course, there is a unique $1$-partition of $m$, namely $\boldsymbol{m_1}=(m)$.  Thus the first direct sum in $\Omega_\ast(A,m)$  coincides with $A_m$. The homology in degree $n$ of $\Omega_\ast(A,m)$ will be denoted either by $T_{n,m}(A)$ or $\mathrm{Tor}_{n,m}^A(R,R)$.  Clearly, we have $T(A)=\oplus_{n,m\in\N} \ T_{n,m}(A)$ and this decomposition is compatible with the coring structure, in the sense that $T(A)$ is a bigraded coring with $T_{n,m}(A)$ in the $(n,m)$ spot.
\end{fact}

\begin{fact}[The Koszul complex $\K_\ast(A)$.] \label{fa:K(A)}Proceeding as in \cite{jps}, to every connected $R$-ring $A$ we associate a chain complex $\K_\ast(A)$ in the category of graded left $A$-modules, that will play the role of the Koszul complex of $A$.
By definition,  we set: $ \dr{K}_n(A)=A \ot A^!_n$.
Let $n > 0$. If  $x=\sum_{i=1}^r x^i_1\ot x_2^i\otimes\cdots\ot  x_r^i$ is an element in $A_n^!$ and $a\in A$, then the map $\pa_n:\K_n(A)\to\K_{n-1}(A)$ satisfies the relation:
\begin{equation*}
	\pa_{n}(a\otimes x)=a\eta_n(x)=\sum_{i=1}^r ax^i_1\ot x_2^i\otimes\cdots\ot  x_r^i.
\end{equation*}
Since $x\in A^!_n$,  it follows that $x$ belongs to $\Ker\mu_{1,1}\ot  A_1^{(n-2)}$. Thus one can see easily that $(\K_\ast (A),\pa_\ast)$ is a complex of  left $A$-modules. 

One can view  $\K_\ast(A)$ as a complex in the category of graded left $A$-modules, with respect to the grading $K_n(A)=\oplus_{p\in \N}K_n(A)_{p}$, where by definition we have $K_n(A)_{p}=A_{p-n}\ot A_n^!$. Of course, by convention, $K_n(A)_{p}=0$, for $p<n$. Let us notice that $K_n(A)$ is a  projective $A$-module which is generated by a set of elements of degree $n$, namely $\{1\ot x\mid x\in A_n^! \}$.

The complex $R\otimes_A \K_\ast(A)$ is isomorphic to  $(A^!_\ast,0)$ via the canonical identification $R\ot_A\K_n(A)\simeq A^!_n$.
\end{fact} 

\begin{fact}[The morphism $\phi^A:A^!\to T(A)$.] \label{fa:exemples_morphisms}

For every $n\in\N$, the component of degree $n$ of $A^!$ is a subset of $ A_1^{( n)}\subseteq A_+^{(n)}$. We claim that the canonical inclusions from $\K_\ast(A)$ to $\beta_\ast(A)$ defines a morphism $\phi'_\ast$ of complexes. Indeed, if $a\ot x \in A\ot A_n^!$, then $\pa_n(a\ot x)=a\eta_n(x)$, by the definition of $\pa_n$. On the other hand, regarding $a\ot x$ as an element in $\beta_n(A)$, we have:
\[
 \delta_n(a\ot x)= a\eta_n(x)+\sum_{i=1}^{n-1}(-1)^i a\ot\big( A_1^{( i-1)}\ot\mu_{1,1}\ot A_1^{( n-i-1)}\big)(x)= a\eta_n(x),
\]
where for the second equality we used the fact that $A_n^!\subseteq A_1^{(i-1)}\ot\Ker\mu_{1,1}\ot A_1^{(n-i-1)}$, for  $0<i<n$. Thus our claim has been proved.

We have remarked that $R\ot_A \K_{\ast}(A)\simeq (A_\ast^!,0)$ and, by definition, $R\ot_A\beta_\ast(A)=\Omega_\ast(A)$. Therefore, the morphism $R\ot_A\phi'_{\ast}:R\ot_A\K_{\ast}(A)\to R\ot_A\beta_{\ast}(A)$ corresponds to a morphism of complexes from $A_\ast^!$ to $\Omega_\ast(A)$, that we denote by $\phi_\ast$. Clearly, $\phi_n$ coincides with the inclusion map of   $A^!_n$ into $\Od{n}{A}$, so $\phi_\ast$ is compatible with  the  graded coring structures of $A^!$ and $\Od{\ast}{A}=T_R^c(A_+)$. Henceforth, passing to homology, $\phi_\ast$ induces a morphism of graded $R$-corings $\phi^A:A^!\to T(A)$.
\end{fact}

\begin{fact}[Koszul $R$-rings.]
By  \cite[Definition 1.1.2]{bgs}, a connected strongly graded $R$-ring is called \emph{Koszul} if there is a resolution $P_{\ast}$ of $R$ by projective graded left $A$-modules such that each $P_{n}$ is generated by its component of degree $n$. 

Our next goal is to is to characterize  Koszul $R$-rings in terms of properties of the $R$-coring $T(A)$. In particular,  we will recover the well known result that an algebra $A$ is Koszul if and only if $T_{n,m}(A)=0$ for $n\neq m$. More precisely, we have the following. 
\end{fact}

\begin{theorem} \label{thm:kring}
Let $A$ be a connected strongly graded $R$-ring. The following are equivalent:
\begin{enumerate}
\item The $R$-ring $A$ is Koszul.
\item The complex $\K_\ast(A)$ is exact in degree $n$, for all $n>0$.
\item The canonical  $R$-coring morphism $\phi^A:A^!\rightarrow T(A)$ is an isomorphism.
\item The $R$-coring $T(A)$ is strongly graded.
\item Any primitive element of $T(A)$ is homogeneous of degree $1$.
\item If $n\neq m$, then $T_{n,m}(A)=0$.
\end{enumerate}	
\end{theorem}

\begin{proof}
Let us assume that $A$ is Koszul. We fix a resolution $(P_\ast,d_\ast)$ of $R$ such that each $P_n$ is graded and generated by $P_{n,n}$, the homogeneous component of degree $n$. Hence, $T_{n,m}(A)$ is the $n$th homology group of the complex $K_\ast(m)=( R\ot_AP_\ast)_m$.  Thus $K_n(m)$ is the quotient of $P_{n,m}$ by the submodule $\sum_{i=1}^{m}A_iP_{n,m-i}$. The differential maps of this complex are induced by those of the resolution $(P_\ast,d_\ast)$. Since $P_n$ is generated by $P_{n,n}$ it follows that  $K_n(m)=0$ for $n\neq m$. Thus  (1)  $\Longrightarrow$ (6).

To prove that $(6) \Longrightarrow (2)$ we proceed by induction, as follows. By construction,  $A_0^!=R$. Thus $\K_0(A)\simeq A$, so we can augment the complex $\K_\ast(A)$ using the projection $\pa_0:A\to R$. 

Since $A_1^!=A_1$ we can identify $\K_1(A)\to \K_0(A)\to R\to 0$ with the sequence: 
\[
 A\ot A_1\rightarrow A \rightarrow R\rightarrow 0,
\]
where the leftmost arrow is induced by the multiplication map of $A$. Using the fact $A$ is strongly graded, it follows that this sequence is exact. Supposing that  the sequence:
\begin{equation}\label{rez1}
A \otimes A_n^! \xrightarrow{\pa_n} A \otimes A_{n-1}^! \xrightarrow{\pa_{n-1}} \cdots \xrightarrow{\pa_1} A\otimes  A_0^! \xrightarrow{\pa_0} R \rightarrow 0
\end{equation}
is exact, we want to prove that $\K_\ast(A)$ is exact in degree  $n$. We denote the module of $n$-cycles by $Z_n$. Thus we have an exact sequence
\begin{equation}\label{sequenceKL}
 0\rightarrow  Z_n\rightarrow A\otimes A_n^! \rightarrow Z_{n-1} \rightarrow 0.
\end{equation}
Since the  homogeneous components of $A\otimes A_n^!$ are trivial  in degree less than $n$ and $d_n$ is injective on $R\ot A_n^!$, the $i$-degree component $Z_{n,i}$ of $Z_n$ is trivial, for any  $i\leq n$.

We complete the sequence \eqref{rez1} to a resolution by projective graded left $A$-modules:
\begin{equation}
\label{rez3} \cdots \rightarrow P_{n+2} \rightarrow P_{n+1} \rightarrow A\ot  A^!_n \rightarrow  A\ot  A^!_{n-1}  \rightarrow \cdots \rightarrow  A\ot  A^!_0 \rightarrow R \rightarrow 0.
\end{equation}
In particular, we obtain a resolution of $Z_{n-1}$ by projective graded left $A$-modules:
\begin{equation}
\label{rez2} \cdots \rightarrow P_{n+2} \rightarrow P_{n+1} \rightarrow  A\ot  A^!_n \rightarrow Z_{n-1} \rightarrow 0.
\end{equation}
Comparing the resolutions (\ref{rez3}) and (\ref{rez2}), we get an isomorphism $\mathrm{Tor}_{n+1,m}^A(R,R) \simeq \mathrm{Tor}^A_{1,m}(R,Z_{n-1})$, for every $m$. But, taking into account the sequence \eqref{sequenceKL}, it follows that  
\[
0 \rightarrow \mathrm{Tor}_{1,m}^A(R,Z_{n-1}) \rightarrow (R \otimes_A Z_n)_m \rightarrow (R\ot A^!_n)_m \rightarrow (R \otimes_A Z_{n-1})_m \rightarrow 0
\] 
is exact, for all $m$. Recall that $R\ot A_n^! \simeq A_n^!$ contains only homogeneous elements of degree $n$. 

We fix $m>n+1$. Since, by hypothesis, $\mathrm{Tor}^A_{n+1,m}(R,R)=0$ it follows that $\mathrm{Tor}^A_{1,m}(R,Z_{n-1})$ vanishes as well. Moreover, using the above exact sequence we get  $(R \otimes_A Z_n)_m=0$.  

As $(A / A_+) \otimes_A  Z_n\simeq Z_n/ A_+Z_n$, we get that $Z_{n,m} =(A_+Z_{n})_m$. We have remarked that $Z_{n,i}=0$, for $i\leq n$. Therefore,
\[
Z_{n,m} = A_{m-n-1}Z_{n,n+1} + \dots + A_1Z_{n,m-1}.
\]
In particular, $Z_{n,n+2}=A_{1}Z_{n,n+1}$. Since $A_{n-m-1}=A_1A_{n-m-2}$, it is not difficult to see by induction on $k>n$ that $Z_{n,k}=A_{k-n-1}Z_{n,n+1}$. Thus $Z_n$ is generated by  $Z_{n,n+1}$, so we conclude the proof of this implication by noticing that $\pa_{n+1}$ induces an isomorphism between the component of degree $n+1$ of $\K_{n+1}(A) $, which coincides with $ R\ot A^!_{n+1}$,   and  $Z_{n,n+1}$.

For the implication $(2) \Longrightarrow (1)$, we remark that, under the standing assumption,  $\dr{K}_\ast(A)$ provides a projective resolution of $R$ that satisfies the conditions in the definition of Koszul rings, cf. \S\ref{fa:K(A)}. 

Let us prove that (2) $\Longrightarrow$ (3). Since $A$ is strongly graded by assumption, it follows that $\K_\ast(A)$ is a resolution of $R$ as a left $A$-module. By \S\ref{fa:exemples_morphisms}, there is a morphism of complexes $\phi _{\ast }':\K_{\ast}(A)\to \beta{}_{\ast}(A)$, which  lifts the identity map of $R$. Since  $\beta{}_{\ast}(A)$ is also a resolution of $R$ by projective left $A$-modules, by the Comparison Theorem,  the morphism $\phi_\ast$  is invertible up to a homotopy in the category of complexes of left $A$-modules. Hence $\phi_\ast=R\ot_A\phi_\ast'$ is  invertible up to a homotopy in the category of complexes of $R$-modules. It follows that $\phi_n^A=H_n(R\ot_A \phi_\ast)$ is an isomorphism for all $n\geq 0$. 

Assuming that the canonical morphism of $R$-corings  $A^!\to T(A)$ is an isomorphism, it follows that $T(A)$ is strongly graded, as $A^!$ is always so. Thus we have proved that (3) $\Longrightarrow$ (4).

The implications (4) $\Longrightarrow$ (5) and (5) $\Longrightarrow$ (6) follow by Lemma \ref{lemain}(2) and Lemma \ref{lemain}(3), respectively.  Note that, in order to apply the third part of Lemma \ref{lemain}, we use the fact that $A$ is strongly graded, so $T_{1,m}=0$, for all $m\neq 1$. Clearly, $T_{0,m}=0$, for all $m>0$. 
\end{proof}

As a consequence of the theorem let us prove that the direct product of two Koszul rings is Koszul, and conversely. This result will be used in \S\ref{fa:tailins} to show that a poset is Koszul, provided that its Hasse diagram is a planar tiling.

\begin{proposition} \label{suma_koszul}
Let $R$ and $S$ be semisimple rings. We assume that $A$ is a connected $R$-ring and $B$ is a connected $S$-ring. The $R \times S$-ring $A \times B$ is Koszul if and only if $A$ and $B$ are Koszul.    
\end{proposition}
\begin{proof}
Let $e_A$ and $e_B$ denote the canonical orthogonal idempotents in $A\times B$. For a right $A\times B$-module $M$ we will denote  the abelian groups $Me_A$ and $Me_B$ by $M_A$ and $M_B$, respectively. We can see $M_A$ as a right $A$-module with respect to the action $m\cdot a=m(a,0)$. In a similar way, $M_B$ will be regarded as a right $B$-module. Note that we have a natural isomorphism of abelian groups:
\[
 \nu_M:M\to M_A\oplus M_B, \quad \nu_M(m)=(me_A,me_B).
\]
Conversely, if $X$ is a right $A$-module and $Y$ is a right $B$-module, then $X\oplus Y$ is a 
 the right $A\times B$-module with respect to the action given by $(x,y)\cdot (a,b)=(x\cdot a,y\cdot b)$. In particular  $ M_A\oplus M_B$ becomes a right $A\times B$-module and $\nu_M$ is an isomorphism of $A\times B$-modules. 

Furthemore, for $X$ and $Y$ as above, we have natural isomorphisms $(X\oplus Y)_A\simeq  X$ and $(X\oplus Y)_B\simeq  Y$. It follows that the mapping $(X,Y)\mapsto X\oplus Y$ defines an equivalence between the product category of the categories of right $A$-modules and $B$-modules category  and the category of right $A\times B$-modules. 

Since an equivalence of category preserves the projective objects, we deduce that $X\oplus Y$ is projective as an $A\times B$-module, provided that $X$ and $Y$ are projective as a right $A$-module and as a $B$-module, respectively. 

Let $P_\ast\to R$ be a projective resolution in the category of right $A$-modules. We assume that every $P_n$ is an $\N$-graded module and the differential morphisms are graded maps. We pick a similar resolution of $Q_\ast\to S$ of  $S$ by projective $B$-modules. In view of the foregoing remarks, $P_\ast\oplus Q_\ast\to R\times S$ is a  resolution of $R\times S$ by projective graded $A\times B$-modules. Taking into account the isomorphism of graded abelian groups:
\[
 (P_\ast\oplus Q_\ast)\ot_{A\times B}(R\times S)\simeq
 (P_\ast \ot_A R)\oplus (Q_\ast\ot_B S),
\]
we deduce that $\Tor_{n,m}^{A\times B}(R\times S,R\times S)\simeq \Tor_{n,m}^{A}(R,R)\oplus \Tor_{n,m}^{B}(S,S)$. Using the previous theorem we conclude that $A\times B$ is Koszul if and only if $A$ and $B$ are so.
\end{proof}

\section{Koszul graded posets} 
In this section, we will investigate the Koszulity of the incidence ring associated to a finite graded poset. We will also  provide some classes of posets whose incidence rings  are Koszul. All of them will be obtained by adjoining a greatest element to a given subset of a Koszul poset. The Koszulity  of the resulting poset is ensured by imposing  homological restrictions  to a certain module, canonically associated to the given subset. 

\begin{fact}[Incidence (co)rings.] \label{fa:incidence}  Let $(\kal{P},\leq)$ be a finite poset. If $x\leq y$  are comparable elements in $\kal{P}$, then one defines the interval $[x,y]$ to be the set $[x,y]:=\{z\in\kal{P}\mid x\leq z\leq y\}$. In the case when  $x<y$ and $[x,y]=\{x,y\}$ we shall say that $x$ is a \emph{predecessor} of $y$ or that $y$ is a \emph{successor} of $x$. By definition, a \emph{chain} in $[x,y]$ is an increasing  sequence $x_0<x_1<\dots < x_l$ such that $x_0=x$ and $x_l=y$. The chain is said to be \textit{maximal}, provided that $x_i$ is a predecessor of $x_{i+1}$, for all $i$. The number $l$ will be called the \emph{length} of the chain.

In this paper we are interested only in \emph{graded posets}, that is in those posets satisfying the property that for every interval all of its maximal chains have the same length. We write $l([x,y])$ for the length of a maximal chain in $[x,y]$, and we refer to this number as the \emph{length} of $[x,y]$. The set of intervals of length $p$ will be denoted by $\kal{I}_p$. 

Let $\Bbbk$ be a field. For  a finite poset $\kal{P}$  we denote the $\Bbbk$-linear  space having a basis $\kal{B}:=\{e_{x,y}\mid x\leq y\}$ by  $\Bbbk[\kal{P}]$. With respect to the product given by: 
\[
e_{x,y} \cdot e_{z,u} = \delta_{y,z} e_{x,u},
\]
 $\Bbbk[\kal{P}]$ becomes an associative and unital $\Bbbk$-algebra, with unit $\sum_{x\in\kal{P}}e_{x,x}$, that will be called the  \emph{incidence algebra} of $\kal{P}$, and it will be denoted by $\Bbbk^a[\kal{P}]$. Note that $\{e_{x,x}\}_{x \in \kal{P}}$ is a set of orthogonal idempotents, so it spans a $\Bbbk$-subalgebra $R$, which is isomorphic to $\Bbbk^n$, where $n=|\kal{P}|$. Hence $\Bbbk^a[\kal{P}]$ can be regarded as an $R$-ring, which is graded and connected, provided that $\kal{P}$ is graded. Note that its homogeneous component of degree $p$ is the linear space generated by all  elements $e_{x,y}$ such that $l([x,y])=p$. The $R$-ring  $\Bbbk^a[\kal{P}]$ will be called the \emph{incidence $R$-ring of} $\kal{P}$.

We also fix  a maximal element $t\in\kal{P}$. Let $\kal{Q}:=\kal{P} \setminus \{t\}$ and let $\kal{F}$ denote the set of predecessors of $t$ in $\kal{P}$. The incidence algebra $\kal{Q}$ will be denoted by $B$. Of course, $B$ is an $S$-ring, where $S=\sum_{x \in \kal{Q}} \Bbbk e_{x,x}$.  
\end{fact}

\begin{fact}[Notation.]\label{fa:not} 
Let $\Bbbk$ be a field, and we fix a finite graded poset $\mathscr{P}$, together with a maximal element $t\in\kal{P}$. Let $\kal{Q}:=\kal{P} \setminus \{t\}$ and let $\kal{F}$ denote the set of predecessors of $t$ in $\kal{P}$. The incidence algebras of $\kal{P}$ and  $\kal{Q}$ will be denoted by $A$ and $B$, respectively. We have seen that  $A$ is an $R$-ring, where $R=\sum_{x \in \kal{P}} \Bbbk e_{x,x}$. Similarly, $B$ is an $S$-ring, where $S=\sum_{x \in \kal{Q}} \Bbbk e_{x,x}$.   Considering only the elements  $e_{x,y}$, with $x,y\in\mathscr{Q}$, we get the canonical basis of $B$.

Finally, let $M$ denote the linear subspace of $A$ generated by all elements $e_{x,t}$, with $x<t$. Hence $e_{x,t}\in M$ if and only if $x \leq u$ for a certain $u \in \kal{F}$. It is easy to see that $M$ is a $B$-submodule of $A$ and
\[ 
 M=\sum_{u \in \kal{F}} Be_{u,t}.
\]
Recall that an $n$-chain in $\kal{P}$  is a sequence $\bs{x}=(x_0,\dots,x_n)$ of elements in $\kal{P}$ such that $x_0<\dots <x_n$. For the element $x_n$ of $\bs{x}$ we will use the special notation $t(\bs{x})$ and we will refer to it as the \textit{target} of  $\bs{x}$. 

The length of $\bs{x}$ is, by definition, the integer $l(\bs{x})=\sum_{i=0}^{n-1} l([x_i,x_{i+1}])=l([x_0,x_n])$, where $l([x,y])$ denotes the length of the interval $[x,y]$. For the set of $n$-chains in $\kal{P}$ of length $m$ we will use the notation $\kal{P}_{n,m}$. 

Let $\kal{P}_n$ denote the set of $n$-chains in $\kal{P}$. 
For every $\bs{x}\in \kal{P}_n$ we define $e_{\bs{x}}\in A_+^{\ot_\Bbbk n}$ by
\[
e_{\bs{x}}:=e_{x_0,x_1} \ot_\Bbbk e_{x_1,x_2} \ot_\Bbbk \dots \ot_\Bbbk e_{x_{n-1}, x_n}.
\]
To investigate the Koszulity of $A$ we need the following result.
\end{fact}

\begin{theorem} \label{thm:toruri}
Keeping the notations and the definitions above, the following identity holds true:
\begin{equation}\label{ec:tor}
 \dr{Tor}_{n,m}^A(R,R)\cong\dr{Tor}_{n,m}^B(S,S) \op \dr{Tor}_{n-1,m}^B(S,M).
\end{equation}	
In particular, $A$ is Koszul if and only if $B$ is so and $\dr{Tor}_{n-1,m}^B(S,M)=0$, for all $n\neq m$.
\end{theorem}

\begin{proof} 
Recall that the index $m$ on the $\dr{Tor}$ spaces corresponds to the internal grading, that is the grading induced by the fact that $\kal{P}$ is a graded poset. By \S\ref{sec-t(a)}, one can compute $\dr{Tor}^A_{n,m}(R,R)$ as the $n$th homology group of the complex $\Od{\ast}{A,m}$. Note that $A_+$ is the direct sum of the $R$-bimodules $\Bbbk e_{x,y}$, where $x$ and $y$ are arbitrary elements in $\kal{P}$ such that $x<y$. Since $\Bbbk e_{x,y}\ot_R \Bbbk e_{x',y'}=0$, whenever $y\neq x'$, we get that 
$\Od{n}{A,m} = (A_+^{ n})_m$ is the direct sum of the $R$-bimodules $ \Bbbk e_{x_0,x_1} \ot_R \Bbbk e_{x_1,x_2} \ot_R \dots \ot_R \Bbbk e_{x_{n-1},x_n}$, where $\bs{x}=(x_0,\dots,x_n)$ is a chain in $\kal{P}_{n,m}$. Moreover, it is not difficult to see  that 
\[
 \Bbbk e_{x_0,x_1} \ot_R \Bbbk e_{x_1,x_2} \ot_R \dots \ot_R \Bbbk e_{x_{n-1},x_n}\cong{} \Bbbk e_{x_0,x_1} \ot_\Bbbk \Bbbk e_{x_1,x_2} \ot_\Bbbk \dots \ot_\Bbbk \Bbbk e_{x_{n-1}, x_n}=\Bbbk e_{\bs{x}}.
\]
for any $n$-chain of length $m$. Thus we can identify $\Omega_{n}(A,m)$ with the $\Bbbk$-linear subspace of the $n$th tensor power of $A_+$ which is spanned by all $e_{\bs{x}}$, with $\bs{x}\in\kal{P}_{n,m}$. Via this identification, the differential of $\Od{\ast}{A,m}$ is given by the formula:
\[ d_n(e_{\bs{x}}) = \sum_{i=1}^{n-1} (-1)^{i-1} e_{x_0,x_1} \ot_\Bbbk \dots \ot_\Bbbk e_{x_{i-1},x_{i+1}} \ot_\Bbbk \dots \ot_\Bbbk e_{x_{n-1}, x_n}.\]
Clearly,  we can identify in a similar way $\Od{n}{B,m}$ with the subspace of $A_+^{\ot_\Bbbk n}$ which  is spanned by the elements $e_{\bs{x}}$, with $\bs{x}\in\kal{Q}_{n,m}$. In this way, $\Od{\ast}{B,m}$ can be seen as a subcomplex of $\Od{\ast}{A,m}$.

Let $\Omega_{n}' $  denote  the linear subspace of $A_+^{\ot_\Bbbk n}$  generated by $e_ {\bs{x}}$, where $\bs{x}\in \kal{P}_{n,m}$ and $t(\bs{x})=t$. Note that $x_0,\dots,x_{n-1}$ are elements in $\kal{Q}$. Obviously, $\Omega_{\ast}' $ is a subcomplex of $\Omega_{\ast}(A,m) $. 

Since $t$ is a maximal element in $\kal{P}$, it follows that $\Omega_\ast(A,m)$ is the direct sum of $\Omega_\ast(B,m)$ and $\Omega_\ast'$. In particular, computing the homology of $\Omega_\ast(A,m)$ we get:
\begin{equation} \label{ec:toruri} 
\dr{Tor}^A_{n,m}(R,R) = \dr{Tor}_{n,m}^B(S,S) \op \dr{H}_n(\Omega'_{\ast}).\end{equation}
To compute the homology group from the above relation, let us note that $\Omega'_{\ast}\cong (\beta^r_{\ast - 1}(B) \ot_B M)_m$, where $\beta^r_\ast(B)$ denotes the normalized bar resolution of $S$ in the category of right $B$-modules.  Indeed,
\[
 \beta^r_{n - 1}(B) \ot_B M\cong (B_+^{\ot_S n-1}\ot_S B)\ot_B M \cong B_+^{\ot_S n-1}\ot_S M.
\]
Proceeding as above we identify the homogeneous component of degree $m$ of $B_+^{\ot_S n-1}\ot_S M$ with the subspace of $A_+^{\ot_\Bbbk n}$ generated by $e_{\bs{x}}\ot_\Bbbk e_{t(\bs{x}),t}$, where $\bs{x}\in \kal{Q}_{n-1}$ and $l(\bs{x})+l([t(\bs{x}),t])=m$. Since the latter linear space is precisely $\Omega_{n}'$ we have an isomorphism  $\Omega'_{n}\cong (\beta^r_{n - 1}(B) \ot_B M)_m$. It is easy to see that these isomorphisms are compatible with the differential maps, so they define an isomorphism of complexes. In conclusion,
\[
 \dr{H}_n(\Omega'_{\ast})\cong \dr{H}_n\big((\beta^r_{\ast -1}(B) \ot_B M)_m\big)
\cong\dr{Tor}_{n-1,m}^B(S,M).
\]
This completes the proof of the isomorphism \eqref{ec:tor}. The claim about the Koszulity of $A$ follows by Theorem \ref{thm:kring}.
\end{proof} 

We need another result which will be used when discussing applications.

\begin{lemma} \label{lema:beu}
Keeping the notations and the definitions above, we have
$\dr{Tor}^A_n(R, Ae_{u,u}) = 0$,  for all $u\in\kal{P}$ and $n>0$. Moreover, $\dr{Tor}_{n,m}^A(R,Ae_{u,u}) = 0$, for all $n \geq 0$ and $m \neq n$.
\end{lemma}

\begin{proof}
	Let $n>0$. Remark that $Ae_{u,u}$ is a projective $A$-module, because $e_{u,u}$ is an idempotent in $A$, so $A=Ae_{u,u} \op A(1-e_{u,u})$. Since $\dr{Tor}_{n,m}^A(R,Ae_{u,u})$ is a direct summand in $\dr{Tor}_n^A(R,Ae_{u,u})=0$, it follows that the former $R$-module is also zero, for all $m$.
	
	We are left to proving that $\dr{Tor}_{0,m}^A(R,Ae_{u,u}) = 0$, for all $m > 0$. This follows by the relation $\Tor_{0,m}^A(R,Ae_{u,u})=(R \ot_A Ae_{u,u})_m$ and the following isomorphisms: 
\[ 
	(R \ot_A Ae_{u,u})_m \cong \Big( \frac{Ae_{u,u}}{A_+(Ae_{u,u})}\Big)_m \cong \Big( \frac{Ae_{u,u}}{A_+e_{u,u}}\Big)_m \cong (R e_{u,u})_m=0. \qedhere
\] 
\end{proof}

\begin{remark}
 The relation $\dr{Tor}_{n,m}^A(R,Ae_{u,u}) = 0$ holds for any poset $\kal{P}$ and every $u\in\kal{P}$, provided that  $m \neq n$. Then, keeping the notation from \S\ref{fa:not}, we also have  $\dr{Tor}_{n,m}^B(S,Be_{s,s}) = 0$, for $s\in\kal{Q}$ and $m\neq n$.
\end{remark}

\begin{lemma}\label{le:tor=0}
 Keeping the notation from \S\ref{fa:not}, we have $\dr{Tor}_{n-1,m}^B(S,Be_{u,t})=0$,  for every $u\in\kal{F}$ and for all $m\neq n$.
\end{lemma}  

\begin{proof}
 The sets $\{e_{x,u}\mid x\leq u\}$ and  $\{e_{x,t}\mid x\leq u\}$ are bases of the  vector spaces $Be_{u,u}$ and $Be_{u,t}$, respectively. Hence the unique linear map given by  $e_{x,u} \mapsto e_{x,t}$, for all $x \leq u$, is bijective. On the other hand, one sees easily that this map  is a morphism of $B$-modules.  With respect to the internal grading, since $\deg(e_{x,t})=\deg(e_{x,u})+1$, the above map induces an isomorphism of degree +1 between the graded modules  $Be_{u,u}$ and $Be_{u,t}$. This implies that:
\[ 
\dr{Tor}_{n-1,m}^B(S,Be_{u,t})\simeq \dr{Tor}_{n-1,m-1}^B(S,Be_{u,u}).
\]
We conclude the proof by applying Lemma \ref{lema:beu} and taking also into account the preceding Remark. 
\end{proof}

\begin{fact}[The condition ($\dagger$).]\label{fa:condition} 
 Under the standing notations and assumptions \S\ref{fa:not}, our next goal is to provide examples of graded posets $\kal{P}$ such that the $R$-ring $A$ is Koszul. In order to do that we want to apply Theorem \ref{thm:toruri}, so we need conditions that ensures the vanishing of $\Tor_{n-1,m}^B(S,M)$, for all $m\neq n$. 
 
We will say that the set $\kal{F}$ \textit{satisfies the condition} $(\dagger)$ if and only if either $\kal{F}$ is a singleton, or there is a common predecessor $s$ of all elements in $\kal{F}$ such that $s$ is the infimum of every couple of distinct elements in $\kal{F}$. Note that $s=\inf \{u,v\}$ if and only if:
\begin{equation} \label{ec:conditie_poset} 
Be_{u,t}\cap Be_{v,t}= Be_{s,t}.
\end{equation}
Although somehow cumbersome and unnatural at this point, the use of this condition will become evident throughout the proof of the next result. Moreover, let us mention that $(\dagger)$ is a purely combinatorial condition, depending only on the structure of the poset $\kal{Q}$ and the set $\kal{F}$. 
\end{fact}

\begin{theorem} \label{te:aplicatie}
Let $\kal{P}$ be a graded poset and let $t\in\kal{P}$ denote a maximal element. Let $A$ and $B$ denote the incidence algebras of $\kal{P}$ and $\kal{Q}:=\kal{P} \setminus \{t\}$, respectively. If the set $\kal{F}$ of all predecessors of $t$ in $\kal{P}$ satisfies the condition $(\dagger)$, then $A$ is a Koszul $R$-ring if and only if $B$ is a Koszul $S$-ring.	
\end{theorem}

\begin{proof} 
As we have noticed, in view of Theorem \ref{thm:toruri}, we have to check that $\dr{Tor}_{n-1,m}^B(S,M)=0$, for all $m\neq n$.
We proceed by induction on the cardinality of $\kal{F}$, which we denote by $f$.

For $f=1$, we have $M=Be_{u,t}$, where $u$ is the unique element of $\kal{F}$. Hence we conclude the proof of the basic case by applying Lemma \ref{le:tor=0}. Note that in this case we do not need the condition ($\dagger$).

Assume now that the result is true for any set $\kal{F}'$ of cardinality $f$ and take $\kal{F}$ a set with $f+1$ elements which satisfies the condition ($\dagger$). Let $s$ denote the common predecessor of the elements of $\kal{F}$. We pick $u \in \kal{F}$ and we take $\kal{F}': = \kal{F}\setminus\{u\}$. It is clear that $\kal{F}'$ satisfies the  condition ($\dagger$), with respect to the same $s$. Therefore, if $M':=\sum_{v \in \kal{F}'} Be_{v,t}$, then $\dr{Tor}_{n-1,m}^B(S,M')=0$, for all $n \neq m$.

Let $\overline{M}:=M/M'$. Using the exact sequence:
\[ \dots \to \dr{Tor}_{n-1,m}^B(S,M') \to \dr{Tor}_{n-1,m}^B(S,M) \to \dr{Tor}_{n-1,m}^B(S,\overline{M}) \to \dots,\]
 it suffices to prove that $\dr{Tor}_{n-1,m}^B(S,\overline{M})=0$ and the conclusion would follow. For this, note that we can make the following identifications:
\[ \overline{M} \cong \frac{Be_{u,t}}{Be_{u,t} \cap \displaystyle\sum_{x \in \kal{F}'} Be_{x,t}} \cong \frac{Be_{u,t}}{Be_{s,t}}.\]
The last isomorphism holds true by an immediate generalization of the relation (\ref{ec:conditie_poset}). Indeed, if we take an element $w$ in $Be_{u,t} \cap \sum_{x \in \kal{F}'} Be_{x,t}$, then  $w=\sum_{y \leq u} \alpha_y e_{y,t} = \sum_{x \in \kal{F}'} \sum_{z \leq x} \beta_{z,x}e_{z,t}$. For every $y$ such that $\alpha_y\neq 0$, the element $e_{y,t}$ must also appear in the double sum with a nonzero coefficient. In particular, there must exist $x\in\kal{F}'$ such that $y\leq x$.  Since  $y\leq u$ as well,  we get that  $e_{y,t}\in Be_{u,t}\cap Be_{x,t}=Be_{s,t}$. Thus $w\in Be_{s,t}$. In conclusion,  $Be_{u,t} \cap \sum_{x \in \kal{F}'} Be_{x,t}$ is a submodule of $Be_{s,t}$. The other inclusion is obvious, so the above intersection and $Be_{s,t}$ coincide, as we claimed. 

Let us consider  the following exact sequence:
\[ \dots \to \dr{Tor}_{n-1,m}^B(S,Be_{u,t}) \to \dr{Tor}_{n-1,m}^B(S,\frac{Be_{u,t}}{Be_{s,t}}) \to \dr{Tor}_{n-2,m}^B(S,Be_{s,t}) \to \dots\]
We already know that $ \dr{Tor}_{n-1,m}^B(S,Be_{u,t}) =0$, cf. Lemma \ref{le:tor=0}. To conclude the proof of the Theorem it remains to show that  $ \dr{Tor}_{n-2,m}^B(S,Be_{s,t})=0$.

Proceeding as in the proof of Lemma \ref{le:tor=0}, one show that there is an isomorphism of degree $2$ between the graded $B$-modules $Be_{s,s}$ and $Be_{s,t}$, which is given by  
$e_{x,s} \mapsto e_{x,t}$ for all $x \leq s$. Thus:
\[
\dr{Tor}_{n-2,m}^B(S,Be_{s,t}) \simeq \dr{Tor}_{n-2,m-2}^B(S,Be_{s,s}).
\]
Hence, by Lemma \ref{lema:beu}, we get $\dr{Tor}_{n-2,m}^B(S,Be_{s,t})=0$.
\end{proof}

Recall that the dual poset of $\kal{P}$ coincides, as a set, with $\kal{P}$. On the other hand, $x$ is less than $y$ in the dual poset of  $\kal{P}$ if and only if $x$ is greater than $y$ in  $\kal{P}$. 

For a subset  $\kal{F}\subseteq\kal{P}$, the condition $(\dagger)$ with respect to the dual partial order relation, can be stated as follows. We will say that $\kal{F}$ \textit{satisfies the condition} $(\ddagger)$ if and only if either $\kal{F}$ is a singleton or there is a common successor $s$ of all elements in $\kal{F}$ such that $s=\sup\{u,v\}$, for all $ u \neq v \in \kal{F}$. 

Since an $R$-ring is Koszul if and only if its opposite $R^{op}$-ring is so, working with the dual poset of $\kal{P}$,  we get the theorem below. 
\begin{theorem} \label{te:aplicatie_dual}
Let $\kal{P}$ be a graded poset and let $t\in\kal{P}$ denote a minimal element. Let $A$ and $B$ denote the incidence algebras of $\kal{P}$ and $\kal{Q}:=\kal{P} \setminus \{t\}$, respectively. If the set $\kal{F}$ of all successors of $t$ in $\kal{P}$ satisfies the condition $(\ddagger)$, then $A$ is a Koszul $R$-ring if and only if $B$ is a Koszul $S$-ring.	
\end{theorem}

We are now able to describe an algorithm, based on  Theorem \ref{te:aplicatie} and Theorem \ref{te:aplicatie_dual}, which will help us produce examples of Koszul posets. 

\begin{fact}[Algorithm.] \label{fa:algorithm}
Start with a Koszul poset $\kal{Q}$ (\textit{e.g.} a finite set regarded as a poset with respect to the trivial partial order) and apply one of the constructions below:
\begin{enumerate} 
  \item Choose $u\in \kal{Q}$ and adjoin a new element $t$ to $\kal{Q}$. On  $\kal{P}:=\kal{Q}\cup\{t\}$ take the partial order relation defined such that $t$ is a successor of $u$, but it is not comparable with other elements of $\kal{Q}$. 
 
 \item Choose $u\in \kal{Q}$ and adjoin a new element $t$ to $\kal{Q}$. On  $\kal{P}:=\kal{Q}\cup\{t\}$ take the partial order relation defined such that $t$ is a predecessor of $u$, but it is not comparable with other elements of $\kal{Q}$. 
 
 \item Choose a subset $\kal{F}\subseteq\kal{Q}$ that satisfies the condition $(\dagger)$ and adjoin a new element $t$ to $\kal{Q}$. On  $\kal{P}:=\kal{Q}\cup\{t\}$ take the partial order relation defined such that $t$ is a successor of all elements of $\kal{F}$, but it is not comparable with other elements of $\kal{Q}$. 
 
 \item Choose a subset $\kal{F}\subseteq\kal{Q}$ that satisfies the condition $(\ddagger)$ and adjoin a new element $t$ to $\kal{Q}$. On  $\kal{P}:=\kal{Q}\cup\{t\}$ take the partial order relation defined such that $t$ is a predecessor of all elements of $\kal{F}$, but it is not comparable with other elements of $\kal{Q}$.
 \end{enumerate}
 Repeat any of the constructions (1)--(4) finitely many times, at each iteration $\kal{Q}$ being replaced with $\kal{P}$. The output poset is Koszul.
\end{fact}

\begin{fact}[Planar  tilings.] \label{fa:tailins} The first examples that we will discuss are those posets whose Hasse diagrams look like planar tilings with square tiles. We start by considering the simplest posets of this type.

By definition a poset  $\mathcal{T}=\{s_\mathcal{T},u_\mathcal{T},v_\mathcal{T},t_\mathcal{T}\}$ is a \textit{tile} if its elements  are the points of the plane $\mathbb{R}^2$ given by $s_\mathcal{T}=(p,q-1)$, $u_\mathcal{T}=(p-1,q)$, $v_\mathcal{T}=(p+1,q)$ and $t_\mathcal{T}=(p,q+1)$ , for some $p,q\in\mathbb{Z}$. The order relation, is defined such that $s_\mathcal{T}<u_\mathcal{T}<t_\mathcal{T}$ and $s_\mathcal{T}<v_\mathcal{T}<t_\mathcal{T}$, but $u_\mathcal{T}$ and $v_\mathcal{T}$ are not comparable. Note that $\mathcal{T}$ is the set of vertices of a square, which will be regarded as the Hasse diagram of $\mathcal{T}$. 

One can show that $\mathcal{T}$ is Koszul using our algorithm. One starts with the poset $\kal{Q}=\{s_\mathcal{T}\}$, which clearly is Koszul.  Then, applying twice the first construction of the algorithm, one adjoins  $u_\mathcal{T}$ and $v_\mathcal{T}$ to $\kal{Q}$. The resulting poset $\kal{P}=\{s_\mathcal{T},u_\mathcal{T},v_\mathcal{T}\}$ is Koszul. Clearly, in $\kal{P}$ we have $s_\mathcal{T}=\inf\{u_\mathcal{T},v_\mathcal{T}\}$ and $u_\mathcal{T}$ and $v_\mathcal{T}$ are not comparable. In order to adjoin $t_\mathcal{T}$ to $\kal{P}$ and to conclude that $\mathcal{T}$ is Koszul, one uses the third construction, setting $\kal{F}=\{u_\mathcal{T},v_\mathcal{T}\}$. Obviously, this set satisfies the condition $(\dagger)$.

A finite union  of tiles will be called \textit{planar tiling}. By definition, $x$ is a predecessor of $y$ in a planar tiling if and only if $x$ is a predecessor of $y$ in one of its tiles. 

Planar tilings are not Koszul in general. For instance, the poset $\kal{P}$ whose Hasse diagram is represented in Figure~\ref{fig:non-koszul} is not  Koszul, since $\Tor_{2,3}^A(R,R)\neq 0$, where $A$ denotes the incidence ring of $\kal{P}$. 
\begin{figure}[h]
 \[
 \includegraphics{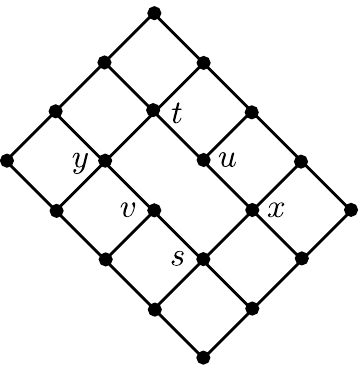} 
\]
\caption{Example of planar tiling which is not Koszul.}\label{fig:non-koszul}
\end{figure}\\
Indeed, $\omega :=e_{s,y}\ot e_{y,t}-e_{s,x}\ot e_{x,t}$ is a $2$-cycle in $\Od{\ast}{A,3}$. Let us assume that  $\omega$ is a $2$-coboundary. Since   $\omega$ has internal degree $3$, it can be written in the following way:
\[
 \omega:=d_3\big (\sum_{(z',z'')\in \kal{S}}\alpha_{z',z''}e_{s,z'}\ot e_{z',z''}\ot e_{z'',t}\big),
\]
where $\kal{S}\subseteq\kal{P}\times\kal{P}$ is the set of couples $(z',z'')$ such that $[s,z']$, $[z',z'']$ and $[z'',t]$ are intervals of length $1$, and $\alpha_{z',z''}\in \Bbbk$. Clearly, $\kal{S}=\{(v,y),(x,u)\}$.  Thus:
\[
 \omega =\alpha_{v,y}(e_{s,v}\ot e_{v,t}-e_{s,y}\ot e_{y,t})+\alpha_{x,u}(e_{s,x}\ot e_{x,t}-e_{s,u}\ot e_{u,t}),
\]
relation that obviously leads us to a contradiction, since the four tensor monomials fom the right-hand side of the above equation are linearly independent over $\Bbbk$.
\end{fact}

\begin{fact}[Koszul planar tilings.]
The union of a Koszul planar tiling $\kal{Q}$ and a tile $\mathcal{T}$ is Koszul as well, provided that we impose additional conditions on the set $\kal{Q}\cap\mathcal{T}$. 

We will consider here four types of sufficient conditions that guarantee the Koszulity of $\kal{P}=\kal{Q}\cup\mathcal{T}$.\vspace*{1mm}

(a) \textit{The case} $\kal{Q}\cap\mathcal{T}=\emptyset.$ Since the intersection of $\kal{Q}$ and $\mathcal{T}$ is empty, we have  $\Bbbk^a[\kal{P}]=\Bbbk^a[\kal{Q}]\times \Bbbk^a[\mathcal{T}]$. Thus, by Proposition \ref{suma_koszul}, $\Bbbk^a[\kal{P}]$ is Koszul.\vspace*{1mm}

(b) \textit{The case} $|\kal{Q}\cap\mathcal{T}|=1.$
Let us assume that  $\kal{Q}\cap\mathcal{T}=\{u_\mathcal{T}\}$. Using the algorithm, one adds successively the missing elements of $\mathcal{T}$, so that at every step one obtains a new Koszul poset. More precisely, one first applies  the construction (1) of the algorithm to add ${t_\mathcal{T}}$ to $\kal{Q}$. Then, one adds $v_\mathcal{T}$ using the construction (3). Finally, one takes $\kal{F}:=\{u_\mathcal{T},v_\mathcal{T}\}$ and uses (4) to patch completely the tile $\mathcal{T}$, by adjoining $s_\mathcal{T}$. Note that the elements in $\kal{F}$ have a unique successor, namely $t_\mathcal{T}$. Henceforth, the condition $(\ddagger)$ is trivially fulfilled for this particular choice of $\kal{F}$.  

If  $\kal{Q}\cap\mathcal{T}:=\{t_\mathcal{T}\}$, then  we first  add $u_\mathcal{T}$ and $v_\mathcal{T}$ to $\kal{Q}$,
using in both cases (3), and then we apply (4) as above. To handle  the cases when $\kal{Q}\cap\mathcal{T}$ is either  $\{v_\mathcal{T}\}$ or $\{s_\mathcal{T}\}$, one proceeds in a similar way. \vspace*{1mm}

(c) \textit{The case} $\kal{Q}\cap\mathcal{T}=\{x,y\}$, where $x$ is a predecessor of $y$ in $\mathcal{T}$. Such a set uniquely determines an edge of the Hasse diagram of $\mathcal{T}$. The fact that $\kal{P}$ is Koszul can be proved using the same method, independently of the choice of $\{x,y\}$. So we only discuss the case when  $x=u_\mathcal{T}$ and $y=t_\mathcal{T}$. First,  using (3), we adjoin the element $v_\mathcal{T}$ to $\kal{Q}$. Then we can repeat the last step from the preceding case to add the remaining element $s_\mathcal{T}$. \vspace*{1mm}  

(d) \textit{The case} $\kal{Q}\cap\mathcal{T}=\{u_\mathcal{T},v_\mathcal{T},x\}$, where $x$ is either $t_\mathcal{T}$ or $s_\mathcal{T}$.  If  $x=t_\mathcal{T}$ then $s_\mathcal{T}$ can be added to $\kal{Q}$ taking $\kal{F}:=\{u_\mathcal{T},v_\mathcal{T}\}$ and using once again (4).  
We have to show  that $\kal{F}$ satisfies the condition $(\ddagger)$. To simplify the notation, we will denote the elements of $\kal{Q}\cap\mathcal{T}$ by $u$, $v$ and $t$. We consider the maximal sequences of tiles $\{\mathcal{T}_1,\mathcal{T}_2,\dots,\mathcal{T}_n\}$ and  $\{\mathcal{T}_1',\mathcal{T}_2',\dots,\mathcal{T}_m'\}$, as in Figure~\ref{fig:sup}.

Keeping the notation from this picture, and assuming that  $x$ in an element in $\kal{Q}$ greater than or equal to $v$, then either $x\geq t$ or there exists $i\in\{0,\dots,m\}$ so that $x=v_i$.  The elements  $x\geq u$ can be characterized in a similar way. Thus an upper bound  $x$ of $\{u,v\}$ must be greater than or equal to $t$, so $t=\sup\{u,v\}$. In particular, a set $\kal{F}$ as above must satisfy the condition $(\ddagger)$. The dual case can be managed  analogously.\begin{figure}[h]  
 \[
\includegraphics{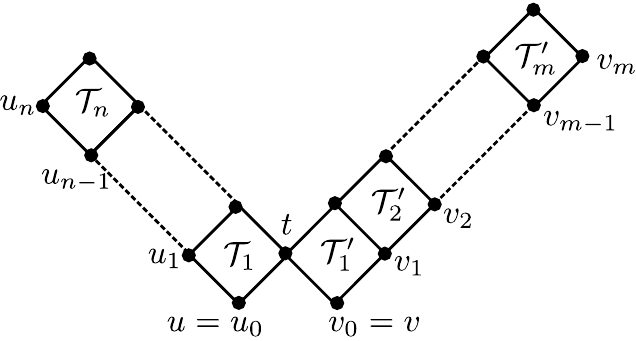}    
\]
\caption{Maximal sequences of tiles sharing one vertex} 
\label{fig:sup}
\end{figure}
\end{fact}

\begin{fact}[An example of planar  tiling.]
To illustrate  the construction of  planar tilings we consider the Hasse diagram from Figure~\ref{fig:tiling}. 
\begin{figure}[h]
 \[
\includegraphics{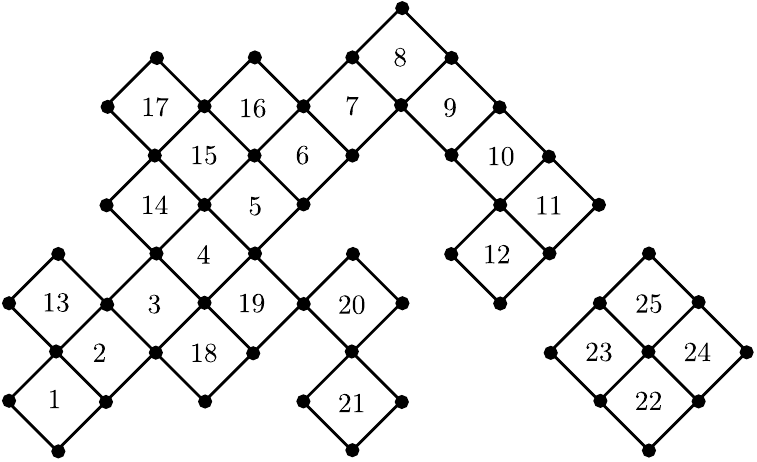}     
\]
\caption{Planar tiling}
\label{fig:tiling}
\end{figure}

One of the possible ways to show that it represents the Hasse diagram of a Koszul poset is to patch step by step the tiles 1 through 8, as in the case (c). Similarly, we can add  the tiles 9 up to 19. 
The adjoining of tiles 20 and 21 follows the case (b), since they intersect the poset previously constructed in one element. On the other hand,
the tile 22 can be patched as having empty intersection with the other component of the poset, while for the tiles 23 and 24 we use the case (c) once again. Finally we add the tile 25 as in (d). 
\end{fact}

\begin{fact}[Nested Diamonds.] In a similar way we get a new Koszul poset, the `nested vertical diamonds', whose Hasse diagram is depicted in the first picture of Figure~\ref{fig:diamonds}.
Start with the trivial poset $\{s\}$ and adjoin the elements $u_1,\dots,u_n$ using the  construction (1) from the algorithm. Thus $s$ is the infimum of each couple $u_i$ and $u_j$ and these elements are not comparable in the resulting poset. Hence we can apply (3), taking $\kal{F}=\{u_1,u_2,\dots,u_n\}$.   

As a last example we consider the poset $\kal{P}_{i,j}$, the `nested horizontal diamonds', represented in the second picture of Figure~\ref{fig:diamonds}. If $i,j>1$, then $\kal{P}_{i,j}$ cannot be constructed using our algorithm, as the infimum and the supremum of  $\kal{F}=\{u,v\}$ do not exist, so the conditions $(\dagger)$ and  $(\ddagger)$ do not hold.
\begin{figure}[h]
\[
 \includegraphics{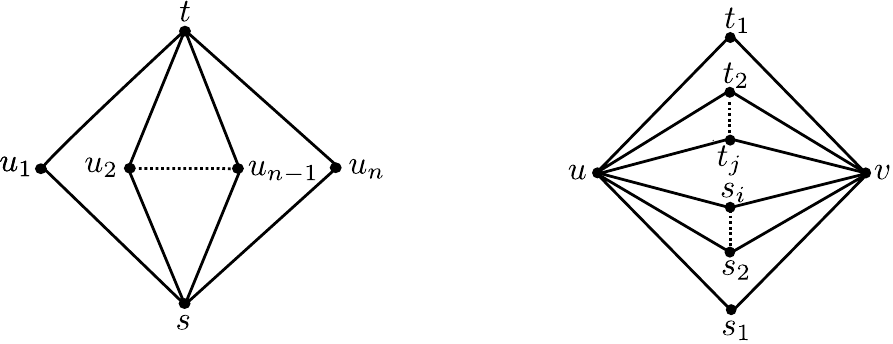}  
\]
\caption{Vertical nested diamonds and horizontal nested diamonds.}\label{fig:diamonds} 
\end{figure}
However, if $\Bbbk$ is a field of characteristic zero, then  $\kal{P}_{i,j}$ is a Koszul poset. To see this, we will prove that $\Bbbk^a[\kal{P}_{i,j}]$ is  a braided symmetric $R$-bialgebra. For the definition of braided $R$-bialgebras in general and of symmetric braided bialgebras in particular, the reader is referred to  \cite[\S6.1]{jps}.  

Let $V$ be the $R$-bimodule generated by all elements $e_{x,y}$ of the canonical basis of $\Bbbk^a[\kal{P}_{i,j}]$, such that $[x,y]$ is an interval of length one. Hence, the set $\{e_{s_p,u}\ot_Re_{u,t_q},e_{s_p,v}\ot_Re_{v,t_q}\mid p\leq i\text{ and } q\leq j\}$ is a basis of $V \ot_R V$ as a linear space. Moreover, since there are no maximal chains of length greater than $2$, it follows that  $V^{\ot_R l}=0$, for all $l>2$. Further, we define the $R$-bilinear braiding $\mathfrak{c} : V \ot_R V \to V \ot_RV$ such that it `interchanges' the chains of length two that correspond to the same $p$ and $q$. More precisely, for all $p\leq i<q$, we have:
\[
 \mathfrak{c}(e_{s_p,u} \ot_R e_{u,t_q})=e_{s_p,v} \ot_R e_{v,t_q}\qquad\text{and}\qquad \mathfrak{c}(e_{s_p,v} \ot_R e_{v,t_q})=e_{s_p,u} \ot_R e_{u,t_q}.
\] 
The linear map $\mathfrak{c}$  satisfies trivially the braid equation, because $V^{\ot_R 3} =0$. Note that  $\mathfrak{c}^2=\Id_{V \ot_R V}$.

Let $I$ denote the ideal generated by the image of $\Id_{V \ot_R V} - \mathfrak{c}$. By definition,  $S_R(V,\mathfrak{c}):=T^a_R(V)/I$. Now we can prove the isomorphism $\Bbbk^a[\kal{P}_{i,j}]\cong S_R(V,\mathfrak{c})$ remarking that the $R$-ring $\Bbbk^a[\kal{P}_{i,j}]$ coincides with the quotient of $T_R^a(V)$ modulo the ideal generated by the differences $e_{s_p,u}\ot_Re_{u,t_q}-e_{s_p,v}\ot_Re_{v,t_q}$, where $p\leq i$ and $q\leq j$.  By \cite[Theorem 6.2]{jps}, it follows that $\Bbbk^a[\kal{P}_{i,j}]$ is a Koszul braided $R$-bialgebra. 
\end{fact}

\begin{remark}
In \cite[Definition 4.6]{wood}, the author defines a graded poset $\Omega$ as being \emph{exactly thin} whenever $x < y$ and $l(x,y)=2$ imply that the interval $(x,y)$ consists of precisely two elements. Note that what we termed `planar tilings' are examples of such posets, which we proved to be Koszul. Furthermore, our `nested horizontal posets' are also exactly thin and Koszul.
\end{remark}

\section{Acknowledgements}

The first named author was financially supported by the project POSDRU/159/1.5/S/137750 of the Ministry of Education-OIPOSDRU.  The second named author was financially supported by CNCS-UEFISCDI, project  PCE PN-II-ID-PCE-2011-3-0635.

\end{document}